\documentclass[11pt,a4paper,reqno]{amsart}
\usepackage[left=30mm,top=40mm,right=30mm,bottom=35mm]{geometry}
\usepackage{epsfig}
\usepackage{amsfonts}
\usepackage{amsmath}
\usepackage{amssymb}
\usepackage{amsthm}
\usepackage{times}
\usepackage{color}
\usepackage{enumerate}
\usepackage[T1]{fontenc}
\usepackage{inputenc}
\usepackage{graphicx}
\usepackage{nicefrac}
\usepackage{xcolor}
\usepackage[all]{xy}
\usepackage{comment}
\usepackage{hyperref}
\hypersetup{
    colorlinks=true,      
    linkcolor=blue,          
    citecolor=blue,       
    filecolor=blue,      
    urlcolor=blue         
}

\usepackage{pgfplots}
\pgfplotsset{compat=1.15}
\usepackage{mathrsfs}
\usetikzlibrary{arrows}

\newtheorem{theorem}{Theorem}[section]
\newtheorem*{theorem*}{Theorem}
\newtheorem{corollary}[theorem]{Corollary}
\newtheorem{lemma}[theorem]{Lemma}
\newtheorem{proposition}[theorem]{Proposition}

{\theoremstyle{remark}
 \newtheorem{remark}[theorem]{Remark}}
{\theoremstyle{definition}
 \newtheorem{definition}[theorem]{Definition}

 \newtheorem{example}[theorem]{Example}

}

\def\Z{\mathbb{Z}}

\def\KK{\mathbf{k}}

\newcommand{\ZZ}[0]{\ensuremath{\mathbb{Z}}}
\newcommand{\GA}[0]{\ensuremath{\mathbb{G}_{\mathrm{a}}}}
\newcommand{\GM}[0]{\ensuremath{\mathbb{G}_{\mathrm{m}}}}
\newcommand{\AF}[0]{\ensuremath{\mathbb{A}}}

\newcommand{\RR}[0]{\ensuremath{\mathbb{R}}}

\newcommand{\QQ}[0]{\ensuremath{\mathbb{Q}}}

\newcommand{\ra}{\rangle}
\newcommand{\la}{\langle}

\newcommand{\pr}[0]{\ensuremath{\operatorname{pr}}}

\newcommand{\spec}[0]{\ensuremath{\operatorname{Spec}}}

\newcommand{\Aut}[0]{\ensuremath{\operatorname{Aut}}}

\newcommand{\rank}[0]{\ensuremath{\operatorname{rank}}}
\newcommand{\cone}[0]{\ensuremath{\operatorname{cone}}}
\newcommand{\homo}[0]{\ensuremath{\operatorname{Hom}}}

\makeatletter \newcommand*\bigcdot{\mathpalette\bigcdot@{.5}} \newcommand*\bigcdot@[2]{\mathbin{\vcenter{\hbox{\scalebox{#2}{$\m@th#1\bullet$}}}}} \makeatother

\begin{document}

\title[On the automorphism group of non-necessarily normal affine toric varieties]{On the automorphism group of \\ non-necessarily normal affine toric varieties}

\author{Roberto D\'iaz}
\address{Instituto de Investigaci\'on Interdisciplinaria, Universidad de Talca,
 Casilla 721, Talca, Chile.}%
\email{robediaz@utalca.cl}

\author{Alvaro Liendo} %
\address{Instituto de Matem\'atica y F\'\i sica, Universidad de Talca,
 Casilla 721, Talca, Chile.}%
\email{aliendo@inst-mat.utalca.cl}

\date{\today}

\thanks{
 \mbox{\hspace{11pt}}{\it Key words}: locally nilpotent derivations, commutative monoids, affine toric varieties, automorphism groups.\\
 \mbox{\hspace{11pt}} The first author was also partially supported by CONICYT-PFCHA/Doctorado Nacional/2016-folio 21161165. The second author was partially supported by Fondecyt project 1200502.}

\begin{abstract}
Our main result is the following: let $X$ be a normal affine toric surface without torus factor. Then there exists a non-normal affine toric surface $X'$ with automorphism group isomorphic to the automorphism group of $X$ if and only if $X$ is different from the affine plane. As a tool, we first provide a classification of normalized additive group actions on a non-necessarily normal affine toric variety $X$ of any dimension. Recall that normalized additive group actions on $X$ are in correspondence with homogeneous locally nilpotent derivations on the algebra of regular functions of $X$. More generally, we provide a classification of homogeneous locally nilpotent derivations on the semigroup algebra of a commutative cancellative monoid. 
\end{abstract}

\maketitle


\section*{Introduction}

Let $\KK$ be an algebraically closed field of characteristic $0$. We denote by $\GA$ and $\GM$ the additive group and the multiplicative group over $\KK$, respectively. Furthermore, we let $T$ be the algebraic torus $T\simeq\mathbb{G}^n_m$, for some integer $n>0$. A toric variety is an irreducible variety $X$ containing a torus $T$ as a Zariski open subset and such that the action of $T$ on itself extends to an algebraic action of $T$ on $X$. There is a well known correspondence between normal affine toric varieties of dimension $n$ and strongly convex polyhedral cones in $\mathbb{Q}^{n}$. This correspondence extends to a duality of categories once morphisms are properly defined and, moreover, extends to non-affine toric varieties by taking certain collections on cones called fans \cite{oda1983convex,fulton1993introduction,cox2011toric}. 

On the toric side, toric morphisms are regular maps that restrict to group homomorphisms in the corresponding acting tori \cite[Proposition 1.3.13]{cox2011toric}. If we drop the normality assumption, then there is still a duality of categories between affine toric varieties with toric morphisms and affine monoids with semigroup homomorphisms \cite[Theorem 1.1.17]{cox2011toric}. Recall that an affine monoid is a finitely generated semigroup with identity that admits an embedding into $\ZZ^n$ for some positive integer $n$. These correspondences between toric varieties and combinatorial objects have allowed for toric varieties to turn into an important area of algebraic geometry, mostly as a source of new examples and testing ground for new theories.

Motivated by the study of the Cremona Groups, Demazure first defined toric varieties in its seminal paper \cite{demazure1970sous}. In this paper, Demazure also introduced a combinatorial gadget that classifies regular $\GA$-actions on  toric varieties that are normalized by the acting torus, see also \cite[Section~1]{liendo2021automorphisms} and \cite[Section~1.3]{ArLi12} for the affine case. Nowadays, this gadget is known as the  Demazure roots. The Demazure roots are in one-to-one correspondence with normalized $\GA$-actions on a toric variety $X$ via the derivation obtained from the $\GA$-action $\GA\times X\to X$ regarded as the flow of a vector field on $X$. In the affine case, such derivations correspond to homogeneous locally nilpotent derivations on the ring of regular functions $\KK[X]$ and the Demazure roots correspond to the degree of such derivations. Recall that a derivation $\partial\colon\KK[X]\to\KK[X]$ is called locally nilpotent if for every $f\in \KK[X]$ there exists an integer $n$ such that $\partial^n(f)=0$. The second named author of this article presented in \cite{liendo2010affine} a modern account of this correspondence that has proved being a useful tool in the study of the automorphism groups of affine toric varieties, see for instance \cite{A13,liendo2018characterization, arzhantsev2019infinite, regeta2021characterizing,kraft2021affine,BoGa21}. 

The main result in \cite[Theorem 1.3]{liendo2018characterization} is the following. Let $X$ be an affine toric surface and let $X'$ be a normal affine surface. If the automorphism groups of $X$ and $X'$ are isomorphic as abstract groups, then $X$ and $X'$ are isomorphic as algebraic varieties. Also in \cite[Proposition 6.2]{liendo2018characterization}, it is proven that the assumption of normality of $X'$ can be dropped if $X$ is the affine plane. On the other hand, in \cite[Proposition 9.1]{regeta2021characterizing2} a counter-example is given if we drop the normality assumption with a particular choice of a normal affine toric surface taking the role of $X$. Our main theorem in this paper states that for any normal affine toric surface $X$ without torus factor different from the affine space, there exists a non-normal affine toric surface $X'$ with automorphism group isomorphic to the automorphism group of $X$, see Theorem \ref{main-surface}. 

Our main result is obtained after generalizing the notion of Demazure root described above to the case of cancellative monoids that are non-necessarily affine nor saturated, see Definition \ref{definition-root}. This allows us to provide a full classification of homogeneous locally nilpotent derivations on the semigroup algebras of cancellative monoids generalizing the one given in \cite{liendo2010affine}, see Theorem~\ref{theorem:equivalence}. This generalization is not straightforward and new techniques are introduced along the way.

\medskip

The paper is organized as follows. In Section~\ref{section 1} we collect and review the essential definitions and results on cancellative monoids, locally nilpotent derivation and affine toric varieties.
Section~\ref{section 2} is devoted to the study of homogeneous locally nilpotent derivations on semigroup algebras of cancellative monoids. In Section~\ref{section 3}, we generalize the notion of Demazure root to the case of cancellative monoids. Finally, in Section~\ref{section 4}, we prove our main result concerning affine toric surfaces.

\subsection*{Acknowledgments}
We would like to thank Andriy Regeta for valuable suggestions on applying Lemma~\ref{AZ-aut-toric-surface} to prove Theorem~\ref{main-surface}. We would also like to thanks both referees. Their reports helped us to improve the quality of this paper and correct inaccuracies.

\section{Preliminaries}\label{section 1}
 
In this section we recall the necessary notions of commutative cancellative monoid, homogeneous locally nilpotent derivations, affine toric geometry and related concepts needed for this paper.

\subsection{Cancellative monoids}\label{sec:monoids} 

A semigroup is a non-empty set $S$ with an associative binary operation $+\colon S\times S\rightarrow S$ where $+(m,m')$ will be denoted by $m+m'$. An element $0\in S$ is an identity if for every $m\in S$, the equations $0+m=m$ and $m+0=m$ hold. If such identity exists it is unique. Semigroups with identity are called monoids. A monoid is called commutative if $m+m'=m'+m$ for all $m,m'\in S$. All monoids in this paper will be assumed to be commutative without explicitly mention it. A monoid is cancellative if for every $m,m',m''\in S$ with $m+m'=m+m''$ we have $m'=m''$. An affine monoid is a finitely generated commutative monoid $S$ that admits an  embedding in $\ZZ^{n}$ for some integer $n\geq 0$.

By \cite[Theorem 3.10]{nagy2001special}, a commutative monoid $S$ is cancellative if and only if $S$ can be embedded in a group. Furthermore, a minimal such group $M_S$ can be obtained as follows. Let $\sim$ be the equivalence relation on $S\times S$ given by $(m_1,m_2)\sim(m_1',m_2')$ if and only if $m_1+m'_2=m'_1+m_2$ in $S$.  We denote the class of $(m_1,m_2)$ in $(S\times S)/\sim$ by $m_1-m_2$. Then $M_S=\{m_1-m_2\mid m_1,m_2\in S\}$.
The identity in $M_S$ is $m-m$ for any $m\in S$ and $m_2-m_1$ is the inverse of $m_1-m_2$. An embedding of $S$ in $M_S$ is given by $m\mapsto m-0$. We will always regard $S$ as a subset of $M_S$ via this embedding and, if no confusion arises, we will denote $M_S$ simply by $M$. Let now $N=\homo(M,\ZZ)$ be the group dual to $M$. We define the dual monoid $S^*$ of $S$ as
$$S^*=\{u\in N\mid u(m)\geq 0,\mbox{ for all } m\in S\}\,.$$

A face of a monoid $S$ is a submonoid $F$ satisfying that $m_1+m_2\in F$ implies $m_1,m_2\in F$. The trivial monoid $0$ is always a face of the dual monoid $S^*$. Let now $S$ be a monoid having $0$ as a face. A face of $S$ isomorphic to $\ZZ_{\geq 0}$ is called a ray of $S$. Let $R$ be a ray of $S$, we define the primitive element of $R$ as the unique generator $\rho$ of $R$ as monoid. We denote by $S(1)$ the set of primitive elements of all the rays of $S$. The saturation of a monoid $S$ is the submonoid $S^{sat}$ of $M$ of the elements $a$ verifying $k\cdot a$ is an element in $S$ for some positive integer $k$.

\subsection{Algebras graded by monoids}

Let $S$ be a monoid. An $S$-graded algebra is an algebra $B$ with a direct sum decomposition $B=\bigoplus_{m\in S}B_m$ where $B_m$ are $\KK$-submodules and $B_{m}B_{m'}\subset B_{m+m'}$. The submodules $B_m$ are called the homogeneous components of $B$. We say that $f\in B_m$ is a homogeneous element of degree $m$. Moreover an $S$-graded algebra can also be regarded as an $M_S$-graded algebra by setting $B_m=\{0\}$ whenever $m\in M_S\setminus S$. The most natural $S$-graded algebra that we can associate to a monoid is the semigroup algebra defined as follows. Given a monoid $S$, we let $\chi^m$ be new symbols for every $m\in S$ and we define the semigroup algebra $\KK[S]$ as 
$$\KK[S]=\bigoplus_{m\in S}\KK\cdot \chi^m\quad\mbox{with multiplication rule given by}\quad \chi^m\cdot\chi^{m'}=\chi^{m+m'}\mbox{ and }\chi^0=1\,.$$

\subsection{Homogeneous locally nilpotent derivation}\label{subsection:homogenous derivation}

Let $B$ be a commutative $\KK$-algebra. A $\KK$-de\-riva\-tion on $B$ is a linear map $\partial\colon B\rightarrow B$ such that $\partial(\KK)=0$ and for every $f, g\in B$ we have $\partial(fg)=f\partial(g)+g\partial(f)$. This last condition is known as the Leibniz rule. All our algebras will be $\KK$-algebras and all our derivations will be $\KK$-derivations, hence we will drop $\KK$ from the notation. A derivation on $B$ is said to be locally nilpotent if for every $f\in B$ there exists a non-negative integer $n$ depending on $f$ such that $\partial^{n}(f)=0$, where $\partial^{n}$ denotes the $n$th iteration of $\partial$. 
   
Locally nilpotent derivations have a geometric counterpart. If $X$ is an affine variety and $\KK[X]$ is the ring of regular functions, there exists a correspondence between locally nilpotent derivations on $\KK[X]$ and regular additive group actions on $X$. The correspondence is defined as follows, see \cite[Section~1.5]{freudenburg2006algebraic} for details. Let $\mu\colon \GA\times X \to X$ be a regular $\GA$-action on $X$. We associate a locally nilpotent derivation $\partial$ to $\mu$ via:
$$\partial_{\mu}\colon \KK[X]\to\KK[X] \quad\mbox{defined by}\quad f\mapsto \left[\dfrac{d}{dt}\circ\mu^*(f)\right]_{t=0}\,.$$ 
Furthermore, every
regular $\GA$-action on $X$ arises from such a locally nilpotent derivation $\partial$. The regular $\GA$-action $\mu$ corresponding to $\partial$ is given as the comorphism of the morphism
$$\mu^*\colon \KK[X]\to\KK[t]\otimes \KK[X] \simeq\KK[X][t]\quad \mbox{defined by} \quad f\mapsto \exp(t\partial)(f)=\sum_{i=0}^{\infty} \dfrac{t^i\partial^{i}(f)}{i!}\,.$$

Let now $S$ be a monoid and let $B$ be an $S$-graded algebra. A derivation $\partial\colon B\to B$ is said to be homogeneous if it sends homogeneous elements into homogeneous elements. For every homogeneous $f\in B_m$ such that $\partial(f)\neq 0$ we define the element $\alpha(f)\in M_S$ such that $\partial(f)\subset B_{m+\alpha(f)}$. In the following lemma we show that $\alpha(f)$ does not depend on the choice of $f$ and so will be denoted simply by $\alpha$. The element $\alpha\in M_S$ is called the degree of $\partial$.

\begin{lemma}\label{lemma:graded}
Let $B$ be an $S$-graded algebra, where $S$ is a cancellative monoid. If $\partial$ is  a homogeneous derivation $\partial\colon B\to B$, then $\alpha(f)\in M_S$ defined above does not depend on  $f\in B\setminus \ker \partial$.
\end{lemma}
\begin{proof}
Let $f\in B_{m}$ and $g\in B_{m'}$ such that $\partial(f)\neq0$ and $\partial(g)\neq 0$. By Leibniz rule we have $$\partial(fg)=f\partial(g)+g\partial(f)\,.$$
The left-hand side is homogeneous of degree $m+m'+\alpha(fg)$ and so the degree of both summands on the right must agree, i.e., $m+m'+\alpha(g)=m+m'+\alpha(f)$ and by the cancellative property we have
$\alpha(f)=\alpha(g)$.
\end{proof}

\subsection{Affine toric variety}\label{sec:toric-varieties}

A (split) algebraic torus $T$ is an algebraic group isomorphic to $\GM^n$ where $n$ is a non-negative integer. There are two mutually dual free abelian groups of rank $n$ canonically associated to $T$, namely, the character lattice $M=\homo(T,\GM)$ and the 1-parameter subgroup lattice $N=\homo(\GM,T)$ with the duality pairing given by $(m,u)\mapsto k$ where $k$ is the unique integer such that $m\circ u(t)=t^k$. It is customary to regard $M$ and $N$ as abstract groups with additive notation and to denote the duality pairing by $\la m,u\ra=u(m)$. This yields that the algebra of regular functions of $T$ corresponds to $\KK[M]$ and the character associated to $m\in M$ corresponds to $\chi^m$.

An affine toric variety is an irreducible affine algebraic variety $X$ containing a torus $T$ as a Zariski open subset such that the action of $T$ on itself by translation extends to an algebraic action of $T$ on $X$. Remark that, following \cite{cox2011toric}, we do not assume normality unlike other authors, see, e.g., \cite{oda1983convex,fulton1993introduction}. The category of affine toric varieties is dual with the category of affine monoids $S$. To obtain the correspondence between objects, if $S$ is an affine monoid, the affine algebraic variety $X_S=\spec(\KK[S])$ is toric with acting torus $T=\spec(\KK[M])$, where $M=M_S$. As for the other direction, let $X$ be an affine toric variety with acting torus $T$. The dominant inclusion $T\subset X$ provides us with an injective homomorphism $\KK[X]\hookrightarrow \KK[M]$ and the  affine monoid $S_X$ associated to $X$ is the monoid $\{m\in M\mid \chi^m\in\KK[X]\}$. 

In Section \ref{section 4} our main interest will be non-normal affine toric varieties. Nevertheless, normal affine toric varieties will also be used as a tool throughout this paper. Letting $S$ be an affine monoid, we let $M=M_S$ and $N=\homo(M,\Z)$. We also define the corresponding rational vector spaces $M_{\QQ}=M\otimes\QQ\simeq \QQ^n$ and $N_{\QQ}=N\otimes\QQ\simeq \QQ^n$. The duality pairing extends naturally to a pairing $M_\QQ\times N_\QQ\rightarrow \QQ$. The dual cones 
$$\sigma=\{u\in N_\QQ\mid u(m)\geq 0\ \text{for all}\ m\in S\}\quad\mbox{and}\quad \sigma^{\vee}=\{m\in M_\QQ\mid u(m)\geq 0\ \text{for all}\ u\in \sigma\}$$ 
are the convex polyhedral cones associated to $X_S$. 
By this construction, the monoid $\sigma^\vee\cap M$ coincides with the saturation of $S$. Furthermore, $S$ is saturated if and only if $\KK[S]$ is integrally closed if and only if $X_S$ is normal. Moreover, if $X_S$ is a non-normal affine toric variety, $X_{\sigma^\vee\cap M}$ is the normalization of $X_S$ with normalization morphism $\eta\colon X_{\sigma^\vee\cap M}\to  X_S$ obtained from the inclusion $S\hookrightarrow \sigma^\vee\cap M$. It is customary to denote $X_{\sigma^\vee\cap M}$ simply by $X_\sigma$.

In this paper we follow the notational conventions in \cite{cox2011toric}. In particular,  the orthogonal space of $A\subset M_\QQ$ is 
$$A^\bot=\{u\in N_\QQ\mid u(m)= 0 \mbox{ for all } m\in A\} $$
We let $H_{m}=\{m\}^\bot$ be the hyperplane orthogonal to $m$. Let  $\sigma\subset N_\QQ$ be a polyhedral cone. A face of $\sigma$ is the intersection  $\tau=\sigma \cap H_m$ for some $m\in \sigma^{\vee}$. By $\sigma(1)$ we denote the set of one-dimensional faces of $\sigma$ that we call rays. A ray will be identified and denoted by its primitive element i.e., the first non-trivial vector in $\rho\cap N$. The analogous definitions hold when we exchange the roles of $M$ and $N$.

The submonoid $\{0\}$ is a face of the affine monoid $S$ if and only if $X_S$ is not isomorphic to $Y\times \GM$. In the case of a normal affine toric variety $X_\sigma$ this condition is also equivalent to $\sigma\subset N_\QQ$ having dimension $\rank N$. In this case, we say that $\sigma$ is a full-dimensional cone meaning that the dimension of $\sigma$ equals the rank of $N$. On the other hand, the cone $\sigma^\vee$ is always full-dimensional by construction. This translate into $\sigma$ being a strongly convex cone meaning that $\{0\}$ is a face of $\sigma$.

\subsection{Locally nilpotent derivations on normal affine toric varieties}

Let $X$ be the normal affine toric variety associated to the polyhedral cone $\sigma\in N_\QQ$. We let $S=\sigma^\vee\cap M$ so that the ring of regular functions of $X$ is $\KK[S]$. There exists a combinatorial description of homogeneous locally nilpotent derivations on $\KK[S]$ given in \cite[Theorem 2.7]{liendo2010affine} by the second author that we recall here.

\begin{definition}\label{defi:root}
A vector $\alpha\in M$ is a Demazure root of $S$ if there exists $\rho\in\sigma(1)$ such that $\rho(\alpha)=-1$ and $\rho'(\alpha) \geq 0$, for every $\rho'\in \sigma(1)\setminus \{\rho\}$. The ray $\rho$ will be called the distinguished ray of the root $\alpha$. By $\mathcal{R}(\sigma)$ we will refer to the set of all roots of $\sigma$ and by $\mathcal{R}_{\rho}(\sigma)$ the set of all roots of $\sigma$ with distinguished ray $\rho$, i.e.,
$$\mathcal{R}_{\rho}(\sigma)=\left\{\alpha\in M\mid \rho(\alpha)=-1\mbox{ and } \rho'(\alpha)\geq0\ \text{for all}\ \rho'\in\sigma(1)\mbox{ different from } \rho \right\}$$
\end{definition}

With the previous notation to every Demazure root $\alpha$ of $\sigma$ with distinguished ray $\rho$ we associate a homogeneous locally nilpotent derivation $$\partial_\alpha\colon \KK[\sigma^\vee\cap M]\rightarrow \KK[\sigma^\vee\cap M] \quad\mbox{defined via}\quad   \chi^m\mapsto\rho(m)\chi^{m+\alpha}\,.$$ 
Furthermore, for every non-trivial homogeneous locally nilpotent derivation $\partial$ on $\KK[S]$ there exists a root $\alpha$ of $S$ and $r\in \KK^*$ such that $\partial=r\partial_\alpha$ \cite[Theorem 2.7]{liendo2010affine}.

\section{Homogeneous decomposition of locally nilpotent derivations}\label{section 2}

Let $B$ be a finitely generated $S$-graded algebra with unit element where $S$ is a finitely generated cancellative monoid and let $\partial\colon B\to B$ be a derivation. Since $\partial$ is a linear map, it can be decomposed into homogeneous pieces. In the following proposition, we show that only finitely many such homogeneous pieces are non-trivial and that each such homogeneous piece is again a derivation.

\begin{proposition} \label{finte-decomposition}
Let $B$ be a finitely generated $S$-graded algebra where $S$ is a finitely generated cancellative monoid. Then every derivation $\partial\colon B\to B$ can be decomposed as a finite sum $\partial=\sum\partial_\alpha$ of $M_S$-homogeneous derivations of degree $\alpha \in M_S$.
\end{proposition}
\begin{proof}
Let $\{f_1,\dots, f_n\}$ be a set of homogeneous generators of $B$ as algebra. Letting $\KK^{[n]}=\KK[x_1,\dots,x_n]$, we have an isomorphism between $B$ and $\KK^{[n]}/I$ where $I$ is the kernel of the surjective homomorphism $\varphi\colon \KK^{[n]}\rightarrow B$ given by $\varphi(x_i)\mapsto f_i$.  
The $M_S$-grading on $B$ induces an $M_S$-grading on $\KK^{[n]}/I$. We can lift this $M_S$-grading to $\KK^{[n]}$ by setting $\KK^{[n]}=\bigoplus_{\alpha\in M} \KK^{[n]}_\alpha$ where $$\KK^{[n]}_\alpha=\{g\in\KK[x_1,\dots,x_n]\mid \varphi(g) \mbox{ is homogeneous of degree } \alpha \}.$$ 
Moreover, choosing $\widetilde{f}_i\in \varphi^{-1}(\partial(f_i))$, we can lift $\partial$ to a derivation $\widetilde{\partial}$ satisfying $\partial\circ\varphi=\varphi\circ\widetilde{\partial}$ by setting
$$\widetilde{\partial}\colon \KK^{[n]}\rightarrow\KK^{[n]}\quad\mbox{given by}\quad x_i\mapsto \widetilde{f}_i$$
We can now decompose $\widetilde{f_i}$ as a finite sum of homogeneous elements in $\widetilde{f}_i=\sum_{j\in M} \widetilde{f}_{i,j}$.
For every $\alpha \in M_S$ we define the homogeneous derivation 
$$\widetilde{\partial}_\alpha\colon \KK^{[n]}\to \KK^{[n]}\quad \mbox{given by} \quad x_i\mapsto \widetilde{f}_{i,a_i+\alpha}\,,$$
so that $\sum_\alpha\widetilde{\partial}_\alpha=\widetilde{\partial}$ and all but finitely many $\widetilde{\partial}_\alpha=0$. Moreover, since $\partial\circ\varphi=\varphi\circ\widetilde{\partial}$ we have $\widetilde{\partial}(I)\subset I$ and since $I$ is homogeneous we have $\widetilde{\partial}_\alpha(I)\subset I$. Hence, $\widetilde{\partial}_\alpha$ passes to the quotient $\KK^{[n]}/I$ and so defines a homogeneous derivation $\partial_\alpha\colon B\to B$ of degree $\alpha\in M_S$ satisfying $\sum_\alpha\partial_\alpha=\partial$.
\end{proof}

We now provide examples showing that the hypothesis cancellative in Lemma~\ref{lemma:graded} and the hypotheses cancellative and finitely generated in Proposition~\ref{finte-decomposition} are essential. 

\begin{example}[Cancellative in Proposition~\ref{finte-decomposition}] Let $S=\{0,a,b\}$ be the monoid with sum operation given by the table below.
$$
\begin{array}{| c | c | c | c |}
\hline
+ & 0 & a & b \\ \hline
0 & 0 & a & b \\ \hline
a & a & b & b \\ \hline
b & b & b & b \\ \hline
\end{array}
$$
It is not cancellative since for instance $a+a=a+b$. Moreover, taking $x=\chi^a$ we have  $\chi^b=\chi^{a+a}=x^2$. It is easily seen that 
$$\KK[S]= \KK\oplus \KK\cdot \chi ^a\oplus\KK\cdot\chi^b\simeq \KK[x]/(x^2-x^3)\,.$$

A straightforward verification shows that the map
$$\partial\colon \KK[S] \to \KK[S] \quad\mbox{given by}\quad \chi^a\mapsto \chi^a-\chi^b \mbox{ and } \chi^b\mapsto 0$$
is a derivation. Nevertheless, it is impossible to exhibit $\partial$ as sum of homogeneous derivations since the map 
$$ 
\partial_0\colon \KK[S]\to  \KK[S]  \quad\mbox{given by}\quad \chi^a\mapsto \chi^a \mbox{ and } \chi^b\mapsto 0$$
is not a derivation. Indeed, the map $\partial_0$ does not satisfy Leibniz rule since
$$0=\partial_0(\chi^b)=\partial_0(\chi^a\cdot\chi^a)=2\chi^a\partial(\chi^a)=2\chi^a\chi^a=2\chi^b\neq 0\,.$$
\end{example}

\begin{example}[Finitely generated in Proposition~\ref{finte-decomposition}] 
Let $S$ be the monoid of sequences of non-negative integers with a finite number of non-zero entries, i.e., 
$$S=\{(a_1,a_2,\dots)\mid a_i\in \ZZ_{\geq 0}\mbox{ and all but finitely many } a_i=0\}\,.$$

The semigroup algebra $\KK[S]$ is the polynomial ring in infinitely many variables $$\KK[S]\simeq \KK[x_1,x_2,\dots], \quad \mbox{where}\quad x_i=\chi^{e_i} \,,$$ 
and $e_i$ is the sequence with the $i$th entry equal to one and all other entries equal to $0$. We define the derivation $$\partial\colon \KK[S]\to \KK[S]\quad \mbox{given by}\quad \chi^{e_i}\mapsto \chi^{2e_i} \mbox{ for all } i\in \ZZ_{>0}\,.$$
The homogeneous components of $\partial$ are 
$$\partial_{e_j}\colon \KK[S]\to \KK[S]\quad \mbox{given by}\quad \chi^{e_j}\mapsto \chi^{2e_j}, \mbox{ and } \chi^{e_i}\to 0  \mbox{ for all } i\neq j\,.$$
These homogeneous components are indeed derivations, but there are infinitely many non-zero such homogeneous components.
\end{example}

\begin{example}[Cancellative in Lemma \ref{lemma:graded}]
Let $S=\ZZ_{\geq 0}^4/\sim$ where $\sim$ is the equivalence relation given by 
\begin{align*}
(m_1,m_2,m_3,m_4)\sim (m_1',m_2',m_3',m_4') \quad \mbox{if and only if} \quad & m_1=m_1'\geq 1, m_2=m_2'\geq 1, \mbox{ and } \\
&m_3+m_4=m_3'+m_4'
\end{align*}

By \cite[Theorem 1.5.2]{howie1995fundamentals}, the sum in the monoid $\ZZ_{\geq 0}^4$ induces a binary operation on $S$ making it into a monoid if and only if for every $m,m',m''\in \ZZ_{\geq 0}^4$ such that $m\sim m'$ we have $m+m''\sim m'+m''$. A straightforward verification  shows that this is the case in our example, so $S$ is a commutative monoid.

Let now $e_i$ be the image of the $i$th standard basis vector on $S$. We define a derivation on $\partial\colon \KK[S]\to \KK[S]$ via 
$$\partial(\chi^{e_1})=\chi^{e_1+e_3},\quad \partial(\chi^{e_2})=\chi^{e_2+e_4},\quad\mbox{and}\quad  \partial(\chi^{e_3})=\partial(\chi^{e_4})=0\,.$$
Remark that $\partial$ is indeed homogeneous since by the Leibniz rule we have:
\begin{align*}
    \partial(\chi^{(m_1,m_2,m_3,m_4)})&=\chi^{m_3e_3}\cdot\chi^{m_4e_4}\left(m_1\chi^{(m_1-1)e_1}\chi^{e_1+e_3}\chi^{m_2e_2}+m_2\chi^{m_1e_1}\chi^{(m_2-1)e_2}\chi^{e_2+e_2}\right) \\
    &=m_1\chi^{(m_1,m_2,m_3+1,m_4)}+m_2\chi^{(m_1,m_2,m_3,m_4+1)}\\
    &=
    \begin{cases}
    m_2\chi^{(m_1,m_2,m_3,m_4+1)} & \mbox{if } m_1=0 \\
    m_1\chi^{(m_1,m_2,m_3+1,m_4)} & \mbox{if } m_2=0 \\
    (m_1+m_2)\chi^{(m_1,m_2,m_3+1,m_4)} & \mbox{if } m_1,m_2\geq 1 \\
    \end{cases}
\end{align*}    
In the last case of the above equation we used the fact that in $S$ we have 
$$(m_1,m_2,m_3+1,m_4)=(m_1,m_2,m_3,m_4+1)\quad \mbox{whenever}\quad  m_1,m_2\geq 1\,.$$
    
In Lemma~\ref{lemma:graded} we defined the degree of an $S$-graded derivation when $S$ is cancellative. In this example, $S$ is not cancellative since $e_1+e_2+e_3=e_1+e_2+e_4$ but $e_3\neq e_4$. Furthermore, it is impossible to define the degree since  $\partial(\chi^{e_1})=\chi^{e_1+e_3}$ would give that the degree is $e_3$ while $\partial(\chi^{e_2})=\chi^{e_2+e_4}$ would give that the degree is $e_4$ which is a contradiction. 

In terms of algebras we have
$$\KK[S]\simeq \KK[x,y,z,t]/I\quad \mbox{where}\quad I=(xyz-xyt)\,,$$
and the derivation $\partial$ is given by 
$$\partial=xz\frac{\partial}{\partial x}+yt\frac{\partial}{\partial y}\,.$$
The derivation $\partial$ defines indeed a derivation in the quotient $\KK[S]=\KK[x,y,z,t]/I$ since 
$$\partial(xyz-xyt)=z\partial(xy)-t\partial(xy)=(z-t)\partial(xy)=(z-t)(xyz-xyt)\in I\,.$$
\end{example}

\medskip

Let $S$ be a finitely generated cancellative monoid. Let $\partial\colon \KK[S]\to \KK[S]$ be a locally nilpotent derivation on the semigroup algebra $\KK[S]$. By Proposition~\ref{finte-decomposition},  the derivation $\partial$ can be decomposed as a finite sum $\partial=\sum\partial_\alpha$ of $M_S$-homogeneous derivations of degree $\alpha \in M_S$. Moreover, if $S$ is affine and saturated, by \cite[Lemma~1.10]{liendo2010affine} some $\partial_\alpha$ is again locally nilpotent. In the remaining of this section we prove the same result in the more general context of finitely generated cancellative monoids, see Theorem~\ref{LND-decomposition}. The techniques required for the proof are not straightforward generalizations of the affine and saturated case.

\begin{lemma}\label{remark:saturation}
Let $S$ be a finitely generated cancellative monoid. If $S$ is saturated, then $S=S'\oplus T$ where $S'$ is a saturated affine monoid and $T$ is a finite group. \end{lemma}

\begin{proof}
Let $M=M_S$ be the smallest group containing $S$. By the structure theorem for finitely generated abelian groups, we have $M=L\oplus T$ where $L$ is free and $T$ is finite. Assume now that $m=l+t\in S$ and let $m'=l+t'\in M$ with $l\in L$ and $t,t'\in T$. We have
$$|T|\cdot m'=|T|\cdot (l+t')=|T|\cdot l=|T|\cdot (l+t)=|T|\cdot m\in S \,.$$
By saturation we conclude that $m'\in S$. This yields $S=S'\oplus T$, where $S'$ is the image of $S$ in $L$ under the first projection. Since $L$ is free and finitely generated, we have $S'$ is an affine monoid. Finally, since $S$ is saturated, $S'$ is also saturated,  proving the lemma.
\end{proof}

Let $S$ be a finitely generated cancellative monoid with associated group $M_S=M$. Recall that we defined $N_S=N$ to be the group dual to $M$ and  $S^*=\{u\in N\mid u(m)\geq 0,\mbox{ for all } m\in S\}$ to be the dual monoid in $N$. We now prove the following lemma.

\begin{lemma}\label{lemma:dual-dual}
Let $S$ be a finitely generated cancellative monoid, then 
$$\widetilde{S}=\{m\in M\mid u(m)\geq 0,\mbox{ for all } u\in S^*\}$$
equals the saturation of $S$.
\end{lemma}

\begin{proof}
Let $S$ be a finitely generated cancellative monoid. Then the inclusion $S^{sat}\subseteq \widetilde{S}$ is trivial. To prove the other inclusion, by the structure theorem for finitely generated abelian groups we have that $M=L\oplus T$, where $L$ is free and $T$ is finite. We also let $S'$ be the image of $S$ in $L$ under the first projection. Let now $m\in \widetilde{S}$. Then $m=l+t$ with $l\in L$ and $t\in T$. We have $|T|\cdot m=|T|\cdot l\in L$. Since $S'$ is an affine monoid we have that $l$ belongs to the saturation of $S'$ by \cite[Proposition~1.2.4]{cox2011toric}. Hence, $l+t'\in S^{sat}$ for some $t'$. Now by Lemma~\ref{remark:saturation} we have that $m=l+t$ is also in $S^{sat}$.  
\end{proof}

\begin{corollary}\label{remark:outsaturation}
Let $S$ be a cancellative finitely generated monoid. If $\alpha \in M\setminus S^{sat}$ then for every $m\in S$ there exists a  positive integer $\ell$ such that $m+\ell\alpha \notin S$. 
\end{corollary}

\begin{proof}
Let $\alpha \in M\setminus S^{sat}$. By Lemma~\ref{lemma:dual-dual}, we have $\alpha \notin \widetilde{S}=S^{sat}$. Hence, there exists  $u\in S^*$ such that $u(\alpha)<0$. Let now $m\in S$. Then $u(m+\ell \alpha)=u(m)+\ell u(\alpha)$. Since $u(\alpha)<0$, taking $\ell$ big enough we can assume $u(m+\ell \alpha)<0$ so that $m+\ell \alpha\notin S^{sat}\supset S$.  
\end{proof}

The following lemma is well known. In lack of a good reference, we provide a proof.

\begin{lemma}\label{lemma:nilradical}
If $T$ is a finite commutative group and $f\in \KK[T]$ is nilpotent, then $f=0$.
\end{lemma}
\begin{proof}
By \cite[Corollary 12.8]{alperin1995group}, $\KK[T]$ is semisimple. Moreover, since $\KK$ is algebraically closed, by Wedderburn's Theorem \cite[Theorem 13.18]{alperin1995group} (see also \cite[Section~18.2]{dummit2004abstract}) the semi\-simple $\KK$-algebra $\KK[T]$ is isomorphic to a direct sum of matrix algebras over $\KK$. Hence, $\KK[T]\simeq \mathcal{M}_{n_1}(\KK)\oplus\cdots\oplus \mathcal{M}_{n_r}(\KK)$ with $r,n_1,\dots,n_r\in \mathbb{Z}_{>0}$ where addition and multiplication are taken component-wise. By the commutativity of $\KK[T]$ we have $n_1=\dots=n_r=1$ i.e., $\KK[T]\simeq \KK\oplus\dots\oplus \KK$. We conclude that if $f\in\KK[T]$ is such that $f^n=0$ for some $n\in \mathbb{Z}_{>0}$ then $f=0$.
\end{proof}

\begin{example}
  The above proof may seem counter-intuitive to those unfamiliar with Wedderburn's Theorem so we provide as an example an isomorphism between $\KK[\ZZ/3\ZZ]$ and $\KK^3=\KK\times\KK\times\KK$ with component-wise addition and multiplication. Let $\KK[\ZZ/3\ZZ]=\KK\oplus\KK\chi^{\overline{1}}\oplus\KK\chi^{\overline{2}}$. Since $\KK$ is algebraically closed of characteristic zero, all three cubic roots of unity $\omega_1$, $\omega_2$ and $\omega_3$ are in $\KK$. Taking $f_i=\tfrac{1}{3}(1\oplus w_i\chi^{\overline{1}}\oplus w_i^2\chi^{\overline{2}})$ for $i=1,2,3$ gives the desired isomorphism since the set $\{f_1,f_2,f_3\}$ is a basis of $\KK[\ZZ/3\ZZ]$ as a vector space and
  $$f_i\cdot f_j=\begin{cases}
  1 & \mbox{if } i=j \\
  0 & \mbox{if } i\neq j 
  \end{cases}$$
\end{example}

For the proof of our main theorem in this section, we need the following remark showing that in $\KK[S]$ a homogeneous element is never a zero divisor.

\begin{remark}\label{remark:zero-divisor}
By the cancellative property on $S$, the homogeneous elements $\chi^m$ is not a zero divisor, for all $m\in S$. Indeed, assume $\chi^m\cdot f=0$ and take the homogeneous decomposition of $f=\sum_i\lambda_i\chi^{m_i}$. Then $\chi^m\cdot f=\chi^m\cdot \sum_i\lambda_i\chi^{m_i}=\sum_i\lambda_i\chi^{m_i+m}=0$. By the cancellative property, we have that all the $m_i+m$ are different and so $\lambda_i=0$ for all $i$.
\end{remark}

Let $S$ be a finitely generated cancellative monoid with $M_S=L\oplus T$, where $L$ is free and $T$ is finite. The following lemma shows that each $M_S$-graded piece of an $L$-homogeneous locally nilpotent derivation in $\KK[S]$ is also locally nilpotent.

\begin{lemma}\label{Lemma: lnd}
Let $S$ be a finitely generated cancellative monoid with $M_S=L\oplus T$, where $L$ is free and $T$ is finite. Let also $\partial$ be a locally nilpotent derivation on $\KK[S]$ with homogeneous decomposition $\partial=\sum_{i=1}^r{\partial_i}$, where each $\partial_i$ is a homogeneous derivation of degree $\alpha_i\in M_S$. If $\alpha_i=\alpha+\alpha_i'$ with $\alpha\in L$ fixed and $\alpha_i\in T$ for all $i\in \{1,\ldots,r\}$, then $\partial_i$ is a homogeneous locally nilpotent derivation on $\KK[S]$ for all $i\in \{1,\ldots,r\}$.  
\end{lemma}

\begin{proof}
Before proceeding with the proof, remark that a straightforward computation shows that we can extend every $\partial_i$ to a derivation $\widetilde{\partial}_i\colon \KK[M_S]\to\KK[M_S]$ by the Leibniz rule via $\widetilde{\partial}_i(\chi^m)=\partial_i(\chi^m)$ and $\widetilde{\partial_i}(\chi^{-m})=-\partial_i(\chi^m)\chi^{-2m}$ for all $m\in S$. Similarly, $\partial$ is extended to $\widetilde{\partial}\colon \KK[M_S]\to\KK[M_S]$ verifying that $\widetilde{\partial}=\sum_{i=1}^r{\widetilde{\partial}_i}$, see also \cite[Theorem 2.3]{Alekseev20}. 

Now, letting $m\in T$, we let $r$ be the smallest integer such that $rm=0$. Then $\widetilde{\partial}(\chi^{rm})=r\cdot\chi^{(r-1)m}\cdot\widetilde{\partial}(\chi^m)=0$. Now, by Remark~\ref{remark:zero-divisor} we conclude that $\widetilde{\partial}(\chi^{m})=0$, for all $m\in T$. We conclude that $\chi^m\in \ker\widetilde{\partial}$ whenever $m$ is a torsion element.

If $\alpha_i\notin S^{sat}$ for all $i\in\{1,\ldots,r\}$, then by Corollary~\ref{remark:outsaturation} for every $i\in\{1,\ldots,r\}$ and every $m\in S$ there exist $\ell$ such that $m+\ell\alpha_i\notin S$ thus $\partial_i^{\ell}(\chi^m)=0$ and so $\partial_i$ is locally nilpotent for all $i\in\{1,\ldots,r\}$.

Assume now that $\alpha_i\in S^{sat}$ for some $i$. We will show that in this case $\partial$ is the trivial derivation. By Lemma~\ref{remark:saturation}, we have that $\alpha$ is also in $S^{sat}$. On the other hand, since each $\partial_i$ is homogeneous of degree $\alpha_i=\alpha+\alpha'_i$, we have that $\widetilde{\partial}_i(\chi^{m})=\lambda_i(m)\chi^{m+\alpha+\alpha'_i}$ with $\lambda_i\colon M_S\to \KK$. The Leibniz rule implies that $\lambda$ is a group homomorphism.  A straightforward verification by induction show
$$\widetilde{\partial}^{n}(\chi^{k\alpha})=\left(\prod_{i=0}^{n-1}(k+i)\right)\chi^{k\alpha+n\alpha}(\lambda_1(\alpha)\chi^{\alpha_1'}+\dots+\lambda_r(\alpha)\chi^{\alpha_r'})^{n},\quad \mbox{for all } n,k\geq 0\,.$$ 
Now choosing $k$ such that $k\alpha\in S$ and $n$ such that $\partial^n(\chi^{k\alpha})=\widetilde{\partial}^n(\chi^{k\alpha})=0$, we have
$$\left(\prod_{i=0}^{n-1}(k+i)\right)\chi^{k\alpha+n\alpha}(\lambda_1(\alpha)\chi^{\alpha_1'}+\dots+\lambda_r(\alpha)\chi^{\alpha_r'})^{n}=0\,.$$ 
By Remark~\ref{remark:zero-divisor} we have $(\lambda_1(\alpha)\chi^{\alpha_1'}+\dots+\lambda_r(\alpha)\chi^{\alpha_r'})^{n}=0$ and by Lemma~\ref{lemma:nilradical}, we have $\lambda_1(\alpha)\chi^{\alpha_1'}+\dots+\lambda_r(\alpha)\chi^{\alpha_r'}=0$. This yields $\widetilde{\partial}(\chi^\alpha)=0$ and moreover, $\lambda_i(\alpha)=0$ for all $i\in\{1,\ldots,r\}$. 

Now, taking into account that $\chi^\alpha$ and $\chi^{\alpha'_i}$ are in the kernel of $\widetilde{\partial}$, again by induction we obtain
$$\partial^{n}(\chi^m)=\widetilde{\partial}^{n}(\chi^m)=\chi^{m+n\alpha}(\lambda_1(m)\chi^{\alpha_1'}+\dots+\lambda_r(m)\chi^{\alpha_r'})^{n},\quad \mbox{for all } m\in S \mbox{ and all }n\geq 0\,.$$ 
For every $m\in S$ we can choose $n$ big enough so that $\partial^{n}(\chi^m)=0$. With the same argument as before, we conclude $\lambda_1(m)\chi^{\alpha_1'}+\dots+\lambda_r(m)\chi^{\alpha_r'}=0$ and so  $\partial(\chi^{m})=0$ ending the proof.
\end{proof}

\begin{remark}
    It is a byproduct of the above proof that the degree $\alpha\in M_S$ of a non-trivial homogeneous locally nilpotent derivation  $\partial\colon\KK[S]\to\KK[S]$ satisfies $\alpha\notin S^{sat}$. This generalizes \cite[Corollary~2.9]{liendo2010affine}.
\end{remark}

We can now state and prove our main theorem in this section.

\begin{theorem} \label{LND-decomposition}
Let $S$ be a finitely generated cancellative monoid and let $\partial\colon \KK[S]\to \KK[S]$ be a derivation. Then $\partial$ admits a decomposition into homogeneous pieces $\partial=\partial_1+\partial_2+\ldots+\partial_k$. Furthermore, if $\partial$ is locally nilpotent, then $\partial_i$ is locally nilpotent as well for some $i\in\{1,\ldots,k\}$.
\end{theorem}

\begin{proof}
Let $S$ be a finitely generated cancellative monoid and consider a derivation $\partial\colon \KK[S]\to \KK[S]$. The existence of a decomposition into homogeneous pieces was proven in Proposition~\ref{finte-decomposition}. 

Assume first that $S$ is finite. For every $m\in S$ there exists a minimal $\ell:=\ell(m)\in \ZZ_{\geq 0}$ such that $\ell\cdot m=0$. Hence, applying the derivation to $\chi^{\ell\cdot m}=(\chi^m)^\ell$ we obtain  $\ell\chi^{(\ell-1)m}\partial(\chi^m)=0$. Since the homogeneous element $\chi^{(\ell-1)m}$ is not a zero divisor (see Remark~\ref{remark:zero-divisor})  we obtain $\partial(\chi^m)=0$ and so $\partial$ is the trivial derivation.

Let now $S$ be any finitely generated cancellative monoid. Let $M=M_S$ be the smallest group containing $S$. By the structure theorem for finitely generated abelian groups we have $M=L\oplus T$, where $L$ is free and $T$ is composed by the torsion elements of $M$.  

Assume now that $\partial$ is locally nilpotent and let $\alpha_i$ be the degree of $\partial_i$, for $i\in \{1,\ldots, k\}$. We decompose $\alpha_i=\alpha_i'+\alpha_i''$ with $\alpha_i'\in L$ and $\alpha_i''\in T$. In the vector space  $L\otimes_{\ZZ}\RR$ we take the polytope $\Delta$ obtained as the  convex hull of the set $\{\alpha_i'\mid i=1,\ldots k\}$.

Let $\alpha'$ be a vertex of $\Delta$ and let $\partial'$ be the sum of all the pieces in the  homogeneous decomposition of $\partial$  having degree with  free part equal to $\alpha'$, i.e.,  
$$\partial'=\sum_{\alpha\in I}\partial_i\quad\mbox{with}\quad I=\{i\in \{1,\ldots,k\}\mid \alpha_i'=\alpha'\}\,.$$

Remark now that $\KK[S]$ is also an $L$-graded algebra and that $\partial'$ is an $L$-homogeneous piece of $\partial$. By \cite[Lemma~1.10]{liendo2010affine}, we have that $\partial'$ is locally nilpotent. Finally, Lemma~\ref{Lemma: lnd} yields that each $\partial_i$ with $i\in I$ is locally nilpotent, proving the theorem.
\end{proof}

\section{
Demazure roots of a cancellative  monoid and locally nilpotent derivations}\label{section 3}

In this section we provide a natural generalization of the notion of Demazure roots to the more general setting of cancellative monoids and show how this notion classifies homogeneous locally nilpotent derivations on the corresponding semigroup algebra.

\subsection{Demazure roots of a cancellative monoid}

The following is our definition of Demazure roots of a non-cancellative monoid.

\begin{definition}\label{definition-root} %
Let $S$ be a cancellative monoid and let $M_S=M$ be its associated group. We also let $N=\homo(M,\ZZ)$ and $S^*$ be the dual monoid. An element $\alpha\in M$ is called a Demazure root of $S$ if 
\begin{enumerate}[$(i)$]
    \item There exists $\rho\in S^*(1)$ such that $\rho(\alpha)=-1$, and
    \item The element $m+\alpha$ belongs to $S$ for all $m\in S$  such that  $\rho(m)>0$.
\end{enumerate}
   We say that $\rho$ is the distinguished ray of $\alpha$. We denote the set of roots of $S$ by $\mathcal{R}(S)$ and the set of roots of $S$ with distinguished ray $\rho$ by $\mathcal{R}_{\rho}(S)$.
\end{definition}

As proven in the following lemma, it is enough to check condition $(ii)$ in the above definition for a generating set of $S$.

\begin{lemma}\label{lemma:base} 
Let $S$ be a cancellative monoid. Let $\alpha\in M$ and  $\rho\in S^*(1)$ be such that  $\rho(\alpha)=-1$. If $A$ is a generating set of $S$, then $m'+\alpha$ belongs to $S$ for all $m'\in S$  such that  $\rho(m')>0$ if and only if $m+\alpha$ belongs to $S$ for all $m\in A$  such that  $\rho(m)>0$.
\end{lemma}
\begin{proof}
The only if part is trivial. As for the other direction, let $m'\in S$ be such that $\rho(m')>0$. Since $A$ is a generating set, let  $m'=n_1m_1+\dots+n_lm_l$ with $n_i$  positive integer and $m_i\in A$ for all $i\in\{1,\dots,l\}$.  Up to reordering, we can assume that $\rho(m_1)>0$. Hence, $m_1+\alpha\in S$ and so $m'+\alpha=(m_1+\alpha)+(n_1-1)m_1+n_2m_2+\dots+n_lm_l \in S$.
\end{proof}

\begin{example}
A numerical monoid is an affine monoid $S\subset\mathbb{Z}_{\geq0}$ such that $\mathbb{Z}_{\geq0}\setminus S$ is a finite set. Let $S$ be a numerical monoid different from $\mathbb{Z}_{\geq0}$. We claim that $\mathcal{R}(S)=\emptyset$. Indeed,  $M=\ZZ$ and so $N=\ZZ\rho$, where $\rho$ is the linear map $\rho\colon M\to \ZZ$ given by $\rho(1)=1$. Hence, $S^*(1)=\{\rho\}$ and the only element in $M$ satisfying $\rho(\alpha)=-1$ is $\alpha=-1$. Let now $m$ be the smallest element in $S\setminus\{0\}$. We have $m>1$ since otherwise $S=\ZZ_{\geq 0}$. We have $\rho(m)=m>0$ and $m+\alpha=m-1\notin S$. Hence, $\alpha$ is not a Demazure root proving the claim.
\end{example}

\begin{example}\label{example:Z-Z-2}
Let $S=\mathbb{Z}_{\geq 0}\times \mathbb{Z}/2\mathbb{Z}$. The group associated to $S$ is $M= \mathbb{Z}\times\mathbb{Z}/2\mathbb{Z}$. The dual group $\operatorname{Hom}(M,\mathbb{Z})$ is isomorphic to $\mathbb{Z}$ and under this identification we have  $S^*\simeq \mathbb{Z}_{\geq 0}$. A straightforward computation shows that   $\mathcal{R}(S)=\{(-1,\overline{0}), (-1,\overline{1})\in \mathbb{Z}\times \mathbb{Z}/2\mathbb{Z}\simeq M\}$.

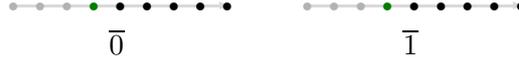
$$
\begin{array}{ll}

\begin{picture}(100,10)
\definecolor{gray1}{gray}{0.7}
\definecolor{gray2}{gray}{0.85}
\definecolor{verde}{RGB}{0,124,0}
\textcolor{gray2}{\put(11,5){\vector(1,0){80}}}

\put(10,5){\textcolor{gray1}{\circle*{3}}} 
\put(20,5){\textcolor{gray1}{\circle*{3}}}
\put(30,5){\textcolor{gray1}{\circle*{3}}}
\put(40,5){\textcolor{verde}{\circle*{3}}}
\put(50,5){\circle*{3}} 
\put(60,5){\circle*{3}}
\put(70,5){\circle*{3}}
\put(80,5){{\circle*{3}}} 
\put(90,5){\circle*{3}}

\end{picture}&\begin{picture}(100,10)

\definecolor{gray1}{gray}{0.7}
\definecolor{gray2}{gray}{0.85}
\definecolor{verde}{RGB}{0,124,0}
\textcolor{gray2}{\put(11,5){\vector(1,0){80}}}

\put(10,5){\textcolor{gray1}{\circle*{3}}} 
\put(20,5){\textcolor{gray1}{\circle*{3}}}
\put(30,5){\textcolor{gray1}{\circle*{3}}}
\put(40,5){\textcolor{verde}{\circle*{3}}}
\put(50,5){\circle*{3}} 
\put(60,5){{\circle*{3}}}
\put(70,5){{\circle*{3}}}
\put(80,5){\circle*{3}} 
\put(90,5){\circle*{3}}
\end{picture}
\\
\hspace{1.9cm}\overline{0}&\hspace{1.9cm}\overline{1}
\end{array}
$$
\end{example}

\begin{example}
Let $S$ be the semigroup $S=\ZZ_{\geq0}^2\times \mathbb{Z}/2\mathbb{Z}\setminus\{(0,1,\overline{0}),(0,0,\overline{1}),(1,0,\overline{1})\}$. The smallest group containing $S$ is $M=\ZZ^2\times \ZZ/2\ZZ$. This yields $N=\mathbb{Z}^2$. The rays are $\rho_1=(1,0)$ and $\rho_2=(0,1)$ in $N$. A direct computation shows that the roots of $S$ are 
\begin{align*}
 \mathcal{R}_{\rho_1}&=\big\{(-1,2+i,j)\mid i\in \ZZ_{\geq 0}\mbox{ and } j\in \ZZ/2\ZZ\big\}\cup \{(-1,1,\overline{1})\}\\
 \mathcal{R}_{\rho_2}&=\big\{(2+i,-1,j)\mid i\in \ZZ_{\geq 0}\mbox{ and } j\in \ZZ/2\ZZ\big\}\cup \{(1,-1,\overline{1})\}
\end{align*}
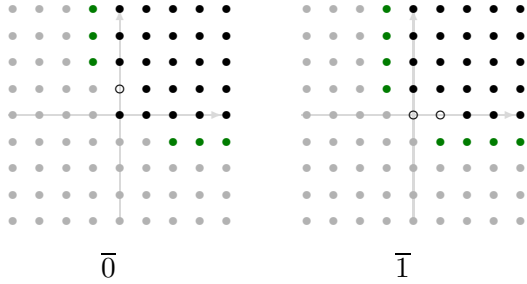
$$
\begin{array}{ll}

\begin{picture}(100,80)
\definecolor{gray1}{gray}{0.7}
\definecolor{gray2}{gray}{0.85}
\definecolor{verde}{RGB}{0,124,0}
\textcolor{gray2}{\put(11,35){\vector(1,0){80}}}
\textcolor{gray2}{\put(50,-5){\vector(0,1){80}}}

\put(10,-5){\textcolor{gray1}{\circle*{3}}} 
\put(20,-5){\textcolor{gray1}{\circle*{3}}}
\put(30,-5){\textcolor{gray1}{\circle*{3}}}
\put(40,-5){\textcolor{gray1}{\circle*{3}}}
\put(50,-5){\textcolor{gray1}{\circle*{3}}} 
\put(60,-5){\textcolor{gray1}{\circle*{3}}}
\put(70,-5){\textcolor{gray1}{\circle*{3}}}
\put(80,-5){\textcolor{gray1}{\circle*{3}}} 
\put(90,-5){\textcolor{gray1}{\circle*{3}}}

\put(10,5){\textcolor{gray1}{\circle*{3}}} 
\put(20,5){\textcolor{gray1}{\circle*{3}}}
\put(30,5){\textcolor{gray1}{\circle*{3}}}
\put(40,5){\textcolor{gray1}{\circle*{3}}}
\put(50,5){\textcolor{gray1}{\circle*{3}}} 
\put(60,5){\textcolor{gray1}{\circle*{3}}}
\put(70,5){\textcolor{gray1}{\circle*{3}}}
\put(80,5){\textcolor{gray1}{\circle*{3}}} 
\put(90,5){\textcolor{gray1}{\circle*{3}}} 

\put(10,15){\textcolor{gray1}{\circle*{3}}} 
\put(20,15){\textcolor{gray1}{\circle*{3}}}
\put(30,15){\textcolor{gray1}{\circle*{3}}}
\put(40,15){\textcolor{gray1}{\circle*{3}}}
\put(50,15){\textcolor{gray1}{\circle*{3}}} 
\put(60,15){\textcolor{gray1}{\circle*{3}}}
\put(70,15){\textcolor{gray1}{\circle*{3}}}
\put(80,15){\textcolor{gray1}{\circle*{3}}} 
\put(90,15){\textcolor{gray1}{\circle*{3}}} 

\put(10,25){\textcolor{gray1}{\circle*{3}}} 
\put(20,25){\textcolor{gray1}{\circle*{3}}}
\put(30,25){\textcolor{gray1}{\circle*{3}}}
\put(40,25){\textcolor{gray1}{\circle*{3}}}
\put(50,25){\textcolor{gray1}{\circle*{3}}} 
\put(60,25){\textcolor{gray1}{\circle*{3}}}
\put(70,25){\textcolor{verde}{\circle*{3}}}
\put(80,25){\textcolor{verde}{\circle*{3}}} 
\put(90,25){\textcolor{verde}{\circle*{3}}}

\put(10,35){\textcolor{gray1}{\circle*{3}}} 
\put(20,35){\textcolor{gray1}{\circle*{3}}}
\put(30,35){\textcolor{gray1}{\circle*{3}}}
\put(40,35){\textcolor{gray1}{\circle*{3}}}
\put(50,35){\circle*{3}} 
\put(60,35){\circle*{3}}
\put(70,35){\circle*{3}}
\put(80,35){{\circle*{3}}} 
\put(90,35){\circle*{3}}

\put(10,45){\textcolor{gray1}{\circle*{3}}}
\put(20,45){\textcolor{gray1}{\circle*{3}}}
\put(30,45){\textcolor{gray1}{\circle*{3}}}
\put(40,45){\textcolor{gray1}{\circle*{3}}}
\put(50,45){{\circle{3}}}
\put(60,45){{\circle*{3}}}
\put(70,45){{\circle*{3}}}
\put(80,45){\circle*{3}} 
\put(90,45){\circle*{3}} 

\put(10,55){\textcolor{gray1}{\circle*{3}}} 
\put(20,55){\textcolor{gray1}{\circle*{3}}}
\put(30,55){\textcolor{gray1}{\circle*{3}}}
\put(40,55){\textcolor{verde}{\circle*{3}}}
\put(50,55){{\circle*{3}}} 
\put(60,55){\circle*{3}}
\put(70,55){\circle*{3}}
\put(80,55){\circle*{3}} 
\put(90,55){\circle*{3}} 

\put(10,65){\textcolor{gray1}{\circle*{3}}} 
\put(20,65){\textcolor{gray1}{\circle*{3}}}
\put(30,65){\textcolor{gray1}{\circle*{3}}}
\put(40,65){\textcolor{verde}{\circle*{3}}}
\put(50,65){{\circle*{3}}}
\put(60,65){{\circle*{3}}}
\put(70,65){\circle*{3}}
\put(80,65){\circle*{3}} 
\put(90,65){\circle*{3}} 

\put(10,75){\textcolor{gray1}{\circle*{3}}} 
\put(20,75){\textcolor{gray1}{\circle*{3}}}
\put(30,75){\textcolor{gray1}{\circle*{3}}}
\put(40,75){\textcolor{verde}{\circle*{3}}}
\put(50,75){{\circle*{3}}} 
\put(60,75){{\circle*{3}}}
\put(70,75){\circle*{3}}
\put(80,75){\circle*{3}} 
\put(90,75){\circle*{3}} 
\end{picture}&\begin{picture}(100,80)
\definecolor{gray1}{gray}{0.7}
\definecolor{gray2}{gray}{0.85}
\definecolor{verde}{RGB}{0,124,0}
\textcolor{gray2}{\put(11,35){\vector(1,0){80}}}
\textcolor{gray2}{\put(50,-5){\vector(0,1){80}}}

\put(10,-5){\textcolor{gray1}{\circle*{3}}} 
\put(20,-5){\textcolor{gray1}{\circle*{3}}}
\put(30,-5){\textcolor{gray1}{\circle*{3}}}
\put(40,-5){\textcolor{gray1}{\circle*{3}}}
\put(50,-5){\textcolor{gray1}{\circle*{3}}} 
\put(60,-5){\textcolor{gray1}{\circle*{3}}}
\put(70,-5){\textcolor{gray1}{\circle*{3}}}
\put(80,-5){\textcolor{gray1}{\circle*{3}}} 
\put(90,-5){\textcolor{gray1}{\circle*{3}}}

\put(10,5){\textcolor{gray1}{\circle*{3}}} 
\put(20,5){\textcolor{gray1}{\circle*{3}}}
\put(30,5){\textcolor{gray1}{\circle*{3}}}
\put(40,5){\textcolor{gray1}{\circle*{3}}}
\put(50,5){\textcolor{gray1}{\circle*{3}}} 
\put(60,5){\textcolor{gray1}{\circle*{3}}}
\put(70,5){\textcolor{gray1}{\circle*{3}}}
\put(80,5){\textcolor{gray1}{\circle*{3}}} 
\put(90,5){\textcolor{gray1}{\circle*{3}}} 

\put(10,15){\textcolor{gray1}{\circle*{3}}} 
\put(20,15){\textcolor{gray1}{\circle*{3}}}
\put(30,15){\textcolor{gray1}{\circle*{3}}}
\put(40,15){\textcolor{gray1}{\circle*{3}}}
\put(50,15){\textcolor{gray1}{\circle*{3}}} 
\put(60,15){\textcolor{gray1}{\circle*{3}}}
\put(70,15){\textcolor{gray1}{\circle*{3}}}
\put(80,15){\textcolor{gray1}{\circle*{3}}} 
\put(90,15){\textcolor{gray1}{\circle*{3}}} 

\put(10,25){\textcolor{gray1}{\circle*{3}}} 
\put(20,25){\textcolor{gray1}{\circle*{3}}}
\put(30,25){\textcolor{gray1}{\circle*{3}}}
\put(40,25){\textcolor{gray1}{\circle*{3}}}
\put(50,25){\textcolor{gray1}{\circle*{3}}} 
\put(60,25){\textcolor{verde}{\circle*{3}}}
\put(70,25){\textcolor{verde}{\circle*{3}}}
\put(80,25){\textcolor{verde}{\circle*{3}}} 
\put(90,25){\textcolor{verde}{\circle*{3}}}

\put(10,35){\textcolor{gray1}{\circle*{3}}} 
\put(20,35){\textcolor{gray1}{\circle*{3}}}
\put(30,35){\textcolor{gray1}{\circle*{3}}}
\put(40,35){\textcolor{gray1}{\circle*{3}}}
\put(50,35){\circle{3}} 
\put(60,35){\circle{3}}
\put(70,35){\circle*{3}}
\put(80,35){{\circle*{3}}} 
\put(90,35){\circle*{3}}

\put(10,45){\textcolor{gray1}{\circle*{3}}}
\put(20,45){\textcolor{gray1}{\circle*{3}}}
\put(30,45){\textcolor{gray1}{\circle*{3}}}
\put(40,45){\textcolor{verde}{\circle*{3}}}
\put(50,45){{\circle*{3}}}
\put(60,45){{\circle*{3}}}
\put(70,45){{\circle*{3}}}
\put(80,45){\circle*{3}} 
\put(90,45){\circle*{3}} 

\put(10,55){\textcolor{gray1}{\circle*{3}}} 
\put(20,55){\textcolor{gray1}{\circle*{3}}}
\put(30,55){\textcolor{gray1}{\circle*{3}}}
\put(40,55){\textcolor{verde}{\circle*{3}}}
\put(50,55){{\circle*{3}}} 
\put(60,55){\circle*{3}}
\put(70,55){\circle*{3}}
\put(80,55){\circle*{3}} 
\put(90,55){\circle*{3}} 

\put(10,65){\textcolor{gray1}{\circle*{3}}} 
\put(20,65){\textcolor{gray1}{\circle*{3}}}
\put(30,65){\textcolor{gray1}{\circle*{3}}}
\put(40,65){\textcolor{verde}{\circle*{3}}}
\put(50,65){{\circle*{3}}}
\put(60,65){{\circle*{3}}}
\put(70,65){\circle*{3}}
\put(80,65){\circle*{3}} 
\put(90,65){\circle*{3}} 

\put(10,75){\textcolor{gray1}{\circle*{3}}} 
\put(20,75){\textcolor{gray1}{\circle*{3}}}
\put(30,75){\textcolor{gray1}{\circle*{3}}}
\put(40,75){\textcolor{verde}{\circle*{3}}}
\put(50,75){{\circle*{3}}} 
\put(60,75){{\circle*{3}}}
\put(70,75){\circle*{3}}
\put(80,75){\circle*{3}} 
\put(90,75){\circle*{3}} 
\end{picture}\\
&\\
\hspace{1.9cm}\overline{0}&\hspace{1.9cm}\overline{1}
\end{array}
$$
\end{example}

\begin{proposition}\label{inclusion}
Let $S$ be a cancellative monoid and let $S^{sat}$ be its saturation. Then every Demazure root of $S$ is also a Demazure root of $S^{sat}$.
\end{proposition}
\begin{proof}
Letting $M$ be the group associated to $S$, we let $\alpha\in M$ be a Demazure root of $S$ with distinguished ray $\rho$. Pick any $m\in S^{sat}$ such that $\rho(m)>0$ and let $k$ be the smallest positive integer such that $k m\in S$. Since  $\rho(km)> 0$ we have $\alpha+km\in S$. If $k=1$ then $\alpha+m\in S^{sat}$ and so the proof is complete. Assume now $k>1$. We have $\rho(km)\geq k$ and so $\rho(\alpha+km)\geq k-1$. Hence $\alpha+(\alpha+k m)=2\alpha+km\in S$ and $\rho(2\alpha+km)\geq k-2$ since $\alpha$ is a Demazure root of $S$. Proceeding inductively, we obtain that $k\alpha+km=k(\alpha+m)\in S$ and by the definition of saturation we conclude $\alpha+m\in S^{sat}$ ending the proof.
\end{proof}

\subsection{Homogeneous locally nilpotent derivation on semigroup algebras}

In the previous section we define a generalization of Demazure roots to the case of cancellative monoids. In this section we prove that, as in the affine toric case, Demazure roots classify locally nilpotent derivations on the semigroup algebra of the corresponding cancellative monoid. 

Our proof is not a straightforward generalization of the affine toric case as one can imagine since many of the usual properties of locally nilpotent derivations are lost in this more general setting where the algebra $B$ is not necessarily a domain. For instance, the kernel of a locally nilpotent derivation $\partial$ on a domain $B$ is factorially closed, i.e., if a product $ab$ belongs to $\ker\partial$ with $a,b\neq 0$, then $a$ and $b$ belong to $\ker\partial$. This is no longer the case if $B$ is not a domain as the following example shows.

\begin{example}
  Let $S=\ZZ_{\geq 0}\times \ZZ/2\ZZ$ as in Example~\ref{example:Z-Z-2}. The corresponding semigroup algebra $B=\KK[S]$ is equal to $\KK[x,y]/(y^2-1)$ by setting $x=\chi^{(1,\overline{0})}$ and $y=\chi^{(0,\overline{1})}$. On $B$ we have the homogeneous locally nilpotent derivation
  $$\partial\colon B\to B\quad\mbox{given by}\quad \partial=\frac{\partial}{\partial x}\,.$$
  Here, the function $f=x(y-1)$ does not belong to the kernel, since $\partial(f)=y-1$. On the other hand, $f\cdot(y+1)=0$ and so $f\cdot(y+1)\in\ker\partial$ .
\end{example}
 
Nevertheless, the above phenomenon does not occur for homogeneous elements as we show in the following Lemma~\ref{lemma: closed}.  For the proof, recall that a homogeneous element is never a zero divisor in $\KK[S]$ as we showed in Remark~\ref{remark:zero-divisor} above.

\begin{lemma}\label{lemma: closed}
Letting $S$ be a cancellative monoid, we let $\partial\colon \KK[S]\to \KK[S]$ be a homogeneous locally nilpotent derivation. If $\partial(f\chi^m)=0$ for some non-zero $f\in\KK[S]$ and some $m\in S$, then $\partial(\chi^m)=0$ and $\partial(f)=0$.
\end{lemma}

\begin{proof}
We follow \cite[Proposition 1.10~(c)]{freudenburg2006algebraic} adapted to our context. Let $m\in S$ and $f\in\KK[S]$ with $f\neq 0$. We let $a$ and $b$ be the smallest non-negative integers such that $\partial^{a+1}(f)=0$ and $\partial^{b+1}(\chi^m)=0$, respectively. We have
$$\partial^{a+b+1}(f\chi^m)=\sum_{i+j=a+b+1}\binom{a+b+1}{i}\cdot\partial^{i}(f)\cdot\partial^{j}(\chi^m)\,.$$
Since $i+j=a+b+1$ then we have either $i>a$ or $j>b$ so $\partial^{i}(f)\partial^{j}(\chi^m)=0$ and so $\partial^{a+b+1}(f\chi^m)=0$. On the other hand, 
$$\partial^{a+b}(f\chi^m)=\sum_{i+j=a+b}\binom{a+b}{i}\cdot\partial^{i}(f)\cdot\partial^{j}(\chi^m)=\binom{a+b}{a}\cdot\partial^{a}(f)\cdot\partial^{b}(\chi^m)\,.$$
The last equality follows with the same argument as above since all the summands are zero with the only exception in the case where $i=a$ and $j=b$. This yields  $\partial^{a+b}(f\chi^m)\neq 0$. Indeed, $\partial^{a}(f)\neq 0$, $\partial^{b}(\chi^m)\neq 0$ and $\partial^{b}(\chi^m)$ is not a zero divisor by Remark~\ref{remark:zero-divisor} since it is a homogeneous element. Since $f\chi^m\neq 0$ and $\partial(f\chi^m)=0$ by hypothesis, we conclude $a+b+1=1$ and so $a=b=0$ which in turn implies $\partial(\chi^m)=\partial(f)=0$ proving the lemma.
\end{proof}

Similarly to the affine toric case, for every Demazure root $\alpha$ of $S$ with distinguished ray $\rho\in S^*(1)$, we define a homogeneous  derivation $$\partial_\alpha\colon \KK[S]\rightarrow \KK[S] \quad\mbox{via}\quad   \chi^m\mapsto\rho(m)\chi^{m+\alpha}\,.$$ 
The degree of $\partial_\alpha$ is $\alpha$. The Leibniz rule follows by a straightforward computation. The fact that $\alpha$ is a Demazure root ensures that $\partial_\alpha$ is well-defined on $\KK[S]$ since $\chi^{m+\alpha}\in \KK[S]$ for every $m$ with $\rho(m)\neq 0$.

 \begin{lemma}\label{homogeneous derivation}
Let $S$ be a cancellative monoid and let $\alpha$ be a Demazure root of $S$. Then, the homogeneous derivation $\partial_{\alpha}$ is locally nilpotent. 
\end{lemma}

\begin{proof}
Let $k$ be a positive integer and let $m\in S$. Let also $\rho$ be the distinguished ray of $\alpha$ so that $\rho(\alpha)=-1$. A straightforward computation shows that
$$\partial^{k+1}(\chi^m)=\rho(m+k\alpha)\cdot \partial^k(\chi^m)\cdot\chi^\alpha\,.$$
Now, taking $k=\rho(m)$ we have
$$\rho(m+k\alpha)=\rho(m)+k\rho(\alpha)=\rho(m)-\rho(m)=0\,.$$
Hence $\partial^{k+1}(\chi^m)=0$ and so $\partial_\alpha$ is locally nilpotent.
\end{proof}

The following lemma characterizes the kernel of a homogeneous locally nilpotent derivation on a semigroup algebra.

\begin{lemma}\label{lemma: face and kernel}
Let $S$ be a cancellative monoid and let $\partial\colon \KK[S]\to \KK[S]$ be a  homogeneous locally nilpotent derivation. Then $\operatorname{ker}\partial=\KK[F]$ where $F$ is a face of $S$.
\end{lemma}
\begin{proof}
Let $\alpha$ be the degree of $\partial$ and let $f=\sum_{i\in I}\lambda_i\chi^{a_i}\in\KK[S]$ with $I$ finite and $\lambda_i\neq 0$. Then we have
$$\partial(f)=\sum_{i\in I}\lambda_i\partial(\chi^{a_i})\,.$$
Since the degree of $\partial(\chi^{a_i})=a_i+\alpha$ and these elements are distinct by the cancellative property, we conclude that $f\in\ker\partial$ if and only if  $\chi^{a_i}\in\ker\partial$ for all $i\in I$.

Let now $F=\{a\in S\mid \chi^{a}\in\ker\partial\}$. The Leibniz rule ensures that $F$ is a subsemigroup of $S$ and the above consideration proves that $\ker\partial=\KK[F]$. Let us now prove that $F$ is a face of $S$. Let $a,b\in S$ and assume that $a+b\in F$. Then  $\partial(\chi^a\chi^b)=\partial(\chi^{a+b})=0$. By Lemma~\ref{lemma: closed} we have  $\partial(\chi^a)=\partial(\chi^b)=0$ which yields $a,b\in F$ proving the lemma.
\end{proof}

\begin{theorem}\label{theorem:equivalence}
Let $\partial\colon \KK[S]\rightarrow \KK[S]$ be a  homogeneous locally nilpotent derivation of degree $\alpha$ and $\partial\neq 0$. Then 
$\alpha$ is a Demazure root of $S$ and $\partial=\lambda\partial_{\alpha}$ for some $\lambda\in \KK^{*}$.
\end{theorem}

\begin{proof}
Since $\partial$ is  homogeneous of degree $\alpha$, we have that
\begin{align}\label{eq:lambda}
    \partial(\chi^m)=\gamma(m)\chi^{m+\alpha}, \quad\mbox{where}\quad  \gamma\colon S\to \KK\,.
\end{align} 
By the Leibniz rule, we have that $\gamma$ is a semigroup homomorphism. Indeed, 
$$\gamma(m+m')\cdot\chi^{m+m'+\alpha}=\partial(\chi^{m+m'})=\partial(\chi^m\cdot\chi^{m'})=(\gamma(m)+\gamma(m'))\cdot\chi^{m+m'+\alpha}\,.$$

Let $M$ be the group associated to $S$. Since $(\KK,+)$ is a group we can extend $\gamma$ to a group homomorphism $\gamma\colon M\rightarrow \KK$ and in turn we can use $\gamma$ to extend $\partial$ to a derivation in $\widetilde{\partial}\colon\KK[M]\to\KK[M]$ given by
$$\widetilde{\partial}\colon \KK[M]\to \KK[M]\quad\mbox{where}\quad  \chi^m\mapsto \gamma(m)\cdot\chi^{m+\alpha}\,.$$ 

Remark now that $\widetilde{\partial}(\chi^{-\alpha})=-\gamma(\alpha)\chi^0\in\KK$ and so $\widetilde{\partial}^{2}(\chi^{-\alpha})=0$ since $\widetilde{\partial}$ is a $\KK$-derivation. We let now $S'$ be the smallest semigroup in $M$ containing $S$ and $-\alpha$. 

By Lemma~\ref{lemma: face and kernel}, we have that $\ker\partial=\KK[F]$, where $F$ is a face of $S$. Let $m\in S$ and let $n$ be the smallest integer such that $\partial^{n+1}(\chi^m)=0$. Then $\partial^{n}(\chi^{m})$ belongs to $\ker\partial$ and it is homogeneous of degree $m_0\in F$. On the other hand, $m=m_0-n\alpha$. Since $\partial$ is non-trivial, we have $\gamma(\alpha)\neq 0$ and so $-\alpha\notin F$. We conclude that 
\begin{align}\label{eq:split}
    S\subset S'=F \oplus (-\alpha)\ZZ_{\geq0},\quad\mbox{and so}\quad M=M_F\oplus \alpha\ZZ\,,
\end{align}
where $M_F$ is the minimal group where $F$ is embedded. Letting now $\lambda=\gamma(-\alpha)$ we let $\rho=\gamma/\lambda$ so that  $\rho(M_F)=0$ and $\rho(\alpha)=-1$.  We have that $\rho\colon M\to \ZZ$ is an element in $N=\homo(M,\ZZ)$. Furthermore, by \eqref{eq:split} we have $\rho\ZZ_{\geq 0}=S^*\cap F^\bot$ and so $\rho$ is a ray of $S^*$ with $\rho(\alpha)=-1$. This provides condition $(i)$ in Definition~\ref{definition-root} of a Demazure root. To prove condition $(ii)$, remark that $m\in S$ satisfies  $\rho(m)>0$ if and only if $\chi^m\notin\ker\partial$. Hence, $m+\alpha\in S$ since otherwise $\partial$ is not a well defined derivation. This proves that $\alpha$ is a Demazure root of $S$. To concludes the proof remark that $\partial=\lambda\partial_\alpha$ by \eqref{eq:lambda}.
\end{proof}

\begin{remark}\label{Remark:correspondence}
    Locally nilpotent derivations on a $\KK$-algebra $B$ are in one-to-one correspondence with additive group actions on $X=\spec B$. This is most often proved for $B$ a domain as in \cite{freudenburg2006algebraic} but it also holds in our context as proven in \cite[Theorem~4.12]{daigle2003}.
\end{remark}

\begin{remark}
The previous theorem could be combined with the results in \cite{BoGa21} to strengthen the results therein. In particular, it looks like the proof of \cite[Lemma~4]{BoGa21} could be significantly simplified applying our Theorem~\ref{theorem:equivalence}.
\end{remark}

It is well known that the ring of invariants of a locally nilpotent derivation on a finitely generated $\KK$-algebra is not necessarily finitely generated itself, see \cite{roberts1990infinitely,daigle1999counterexample}. In contrast, we have the following corollary.
 
 \begin{corollary}
Let $S$ be a finitely generated cancellative monoid and let $\partial\colon \KK[S]\to \KK[S]$ be a homogeneous locally nilpotent derivation. Then  $\ker\partial$ is a finitely generated $\KK$-algebra.
\end{corollary}

\begin{proof}
Recall that a semigroup algebra $\KK[S]$ is finitely generated if and only if $S$ is finitely generated as a semigroup. The Corollary now follows directly from Lemma~\ref{lemma: face and kernel} since every face of a finitely generated semigroup is finitely generated itself.
\end{proof}

\section{Automorphism group of non-necessarily normal affine toric surface}\label{section 4}

Our original motivation to study cancellative monoids is to study automorphism groups of non-necessarily normal affine toric varieties. Our main result in this paper is that for every normal affine toric surface $X_\sigma$ different from $\GM^2$, $\GM\times \AF^1$ and $\AF^2$, there exists a non-normal affine toric surface $X_S$ whose normalization is $X_\sigma$ such that $\Aut(X_\sigma)\simeq \Aut(X_S)$.

We begin by showing that both our definitions of Demazure roots coincide when $S$ is a saturated affine monoid, i.e., the monoid associated to a normal affine toric variety. This fact also follows indirectly from \cite[Theorem 2.7]{liendo2010affine} and Theorem~\ref{theorem:equivalence} since these theorems show, in particular, that with both definitions, the Demazure roots corresponds to the weights of homogeneous locally nilpotent derivations on $\KK[S]$. Nevertheless, we provide a direct proof of this fact. 

Let $S$ be an affine monoid. Recall that $M=M_S$ is the smallest group where $S$ can be embedded and $N=N_S=\homo(M,\ZZ)$. The cone associated to $S$ is $\sigma=\{u\in N_\QQ\mid u(m)\geq 0 \mbox{ for all } m\in S\}$. 

\begin{lemma} \label{def-agree}
Let $S$ an affine monoid. If $S$ is saturated then  $\mathcal{R}(S)=\mathcal{R}(\sigma)$.
\end{lemma}

\begin{proof} 
It is enough to verify that $\mathcal{R}_\rho(S)= \mathcal{R}_\rho(\sigma)$ for any fixed ray $\rho\in \sigma(1)$. Letting $\alpha\in \mathcal{R}_\rho(S)$, we let $\rho'\in\sigma(1)\setminus\{\rho\}$ and we fix $m\in ((\rho')^\bot\cap S)\setminus \rho^\bot$. By the definition of Demazure root, we have $\alpha+m\in S$ so that $\rho'(\alpha)=\rho'(\alpha+m)\geq 0$. Hence, $\alpha\in \mathcal{R}_\rho(\sigma)$ and we have  $\mathcal{R}_\rho(S)\subset \mathcal{R}_\rho(\sigma)$. 

As for the other inclusion, let $\alpha\in\mathcal{R}_\rho(\sigma)$ and let $\rho'\in\sigma(1)\setminus\{\rho\}$. Letting $A$ be a generating set of $S$, we let $m\in A$ with $\rho(m)>0$. Then  $\rho'(m+\alpha)\geq 0$ for all $\rho'\in \sigma(1)$ so $m+\alpha\in \sigma^\vee\cap M=S$ by saturation. This yields $\alpha\in \mathcal{R}_\rho(S)$.
\end{proof}

If $S$ is not saturated, we only have the inclusion $\mathcal{R}(S)\subseteq \mathcal{R}(\sigma)$ and typically this inclusion is proper. For instance, if $X_\sigma$ is the affine space, then the equality never holds by \cite{cantat2019families}. Nevertheless, in Proposition~\ref{roots: non normal}, we will show that for every normal affine toric variety $X_{\sigma}$ without torus factor that is not a product $X_{\sigma'}\times \AF^1$, we can always find a non-normal affine toric variety $X_S$ such that $X_{\sigma}$ is the normalization of $X_S$ and $\mathcal{R}(S)=\mathcal{R}(\sigma)$. Before we do so, we need to prove the following lemma.

\begin{lemma} \label{neg-root-affine}
Let $\sigma$ be a strongly convex cone. Let   $\alpha\in\mathcal{R}$ be a Demazure root. If  $-\alpha\in (\sigma^\vee\cap M)$ then $X_{\sigma}\simeq X_{\sigma'}\times \AF^1$ with $X_{\sigma'}$ an affine toric variety
\end{lemma}

\begin{proof}
Let $\rho$ be the distinguished ray of $\alpha$. Since $-\alpha\in (\sigma^\vee\cap M)$, we have \begin{align} \label{roots-regular}
    \rho'(-\alpha)=\begin{cases}
    1 & \mbox{if }\rho'=\rho \\
    0 & \mbox{if }\rho'\neq\rho
    \end{cases}
\end{align}
Let $N'=\alpha^\bot$ be the orthogonal to the Demazure root $\alpha$. Since $\alpha$ is primitive, we can split $N=N'\oplus \rho\ZZ$. Let now $\sigma'=\cone(\rho'\mid \rho'\neq \rho)$. By \eqref{roots-regular}, we have $\sigma'\subset N'_\QQ$. Now, the cone generated by $\sigma'$ and $\rho$ equals $\sigma$ and moreover, the splitting  $N=N'\oplus \rho\ZZ$ yields $X_\sigma=X_{\sigma'}\times \AF^1$.
\end{proof}

Let $\sigma$ be a full-dimensional strongly convex polyhedral cone $\sigma \in N_\QQ$.  Recall that $\sigma$ is full-dimensional if and only if $X_\sigma$ is non-degenerate, i.e., if $X_\sigma$ does not have a torus factor. Let also $\mathcal{H}$ be the Hilbert basis of $S=\sigma^\vee\cap M$, i.e., the set of element in $S$ that are irreducible in the sense that they cannot be expressed as sum of two non-trivial elements in $S$. Recall that $\mathcal{H}$ is a generating set of the semigroup $\sigma^\vee\cap M$.

\begin{proposition}\label{roots: non normal}
Let $X_\sigma$ be a non-degenerated affine toric variety. If  $S=(\sigma^{\vee}\cap M)\setminus\mathcal{H}$, then $\mathcal{R}(\sigma)= \mathcal{R}(S)$ if and only if $X_\sigma$ cannot be decomposed as a product $X_{\sigma'}\times \AF^1$ with $X_{\sigma'}$ an affine toric variety.
\end{proposition}

\begin{proof} 
If $X_\sigma=X_{\sigma'}\times \AF^1$, then $M=M'\times \ZZ$, $N=N'\times \ZZ$ and $\sigma$ is the cone generated by $(0,1)$ and $(\sigma',0)$, where $\sigma'$ is a cone in $N'_\QQ$ corresponding to the affine toric variety  $X_{\sigma'}$. In this case, a straightforward computation shows that $\alpha=(0,-1)\in M$ is a Demazure root of $\sigma$ with respect to the ray $\rho=(0,1)\in N$. Furthermore,  $(0,1)\in M$ belongs to $\mathcal{H}$. By Definition~\ref{definition-root}, we have that $\alpha$ is not a Demazure root of $S=(\sigma^\vee\cap M)\setminus \mathcal{H}$ since $m=(0,2)\in S$, $\rho(m)=2>0$ and $m+\alpha\notin S$. This yields $\mathcal{R}(\sigma)\neq \mathcal{R}(S)$, proving one direction of the proposition.

To prove the other direction, assume that $X_\sigma$ cannot be decomposed as a product $X_{\sigma'}\times \AF^1$. Let $A$ be a generating set of $S$. Let $\rho\in \sigma(1)$. By Proposition~\ref{inclusion}, we have $\mathcal{R}_{\rho}(S)\subset \mathcal{R}_{\rho}(\sigma)$, so we only need to verify the converse inclusion $\mathcal{R}_{\rho}(\sigma)\subset \mathcal{R}_{\rho}(S)$. A Demazure root $\alpha\in \mathcal{R}_{\rho}(\sigma)$ is contained in $\mathcal{R}_\rho(S)$ if and only if $\alpha+a\in S$ for all $a\in A$ such that $\rho(a)>0$.  Nevertheless,  by Definition~\ref{definition-root} and Lemma~\ref{def-agree} we have $\alpha+a\in \sigma^\vee\cap M$. Hence, $\alpha\in \mathcal{R}_{\rho}(S)$ if and only if $\alpha+a \notin \mathcal{H}$. Assume $\alpha+a\in \mathcal{H}$. We have that $a\notin \mathcal{H}$, hence $a=h_1+\ldots+h_\ell$ with $h_i\in \mathcal{H}$ and $\ell\geq 2$. Since $\rho(a)>0$ we have $\rho(h_i)>0$ for some $i$. Without loss of generality, we may and will assume $\rho(h_1)>0$. Hence $a=h_1+a'$ with $a'\in (\sigma^\vee\cap M)\setminus \{0\}$. But now $\alpha+a=\alpha+h_1+a'$. Again by definition of Demazure root we have $\alpha+h_1\in \sigma^\vee\cap M$. Moreover, if $\alpha+h_1= 0$ the $-\alpha\in \sigma^\vee\cap M$ and by Lemma~\ref{neg-root-affine}, we have that $X_\sigma$ can be decomposed as a product $X_{\sigma'}\times \AF^1$, which is a contradiction. Hence, we may and will assume $\alpha+h_1\neq 0$. This yields  $\alpha+a=(\alpha+h_1)+a'$ with $(\alpha+h_1),a'\in (\sigma^\vee\cap M)\setminus \{0\}$ and so  $\alpha+a\notin\mathcal{H}$ since it is not irreducible proving the proposition.
\end{proof}

Before stating our main theorem in this section, we need the following lemma.

\begin{lemma} \label{lifting-to-normalization}
Let $X$ and $Y$ be algebraic varieties and assume that $X$ is the normalization of $Y$. Then every automorphism of $Y$ lifts to an automorphism of $X$. Moreover, we have $\Aut(Y)\subseteq\Aut(X)$ with equality if and only if every automorphism of $X$ is the lifting of an automorphism of $Y$.      
\end{lemma}

\begin{proof}
    The first statement follows directly from the Universal Property of Normalization. Remark now that the normalization map $X\to Y$ is a birational isomorphism and under this isomorphism, both groups $\Aut(X)$ and $\Aut(Y)$ are naturally subgroups of the group of birational maps of $X$. Now the second statement follows from the first one.  
\end{proof}

Our main interest in this section is affine toric surfaces. For these surfaces, we will prove the following theorem.

\begin{theorem} \label{main-surface}
Let $X_\sigma$ be a normal affine toric surface. Then, there exists a non-normal affine toric surface $X_S$ whose normalization is $X_\sigma$ with $\Aut(X_\sigma)=\Aut(X_S)$ if and only if $X_\sigma$ is different from $\GM^2$, $\GM\times \AF^1$ and $\AF^2$.
\end{theorem}

We need some preliminary results first.

\begin{lemma} \label{no-GA}
Let $X_S$ be an non-normal affine toric surface related with the affine monoid $S$. If the normalization of $X_S$ is isomorphic to  $\GM\times \AF^1$, then $S$ has not Demazure root.
\end{lemma}

\begin{proof}
The affine monoid $S'$ corresponding to the normalization $\GM\times \AF^1$ of $X_S$ is $S'=\ZZ\times \ZZ_{\geq 0}$. The set of Demazure roots of $S'$ is  $\mathcal{R}(S')=\{(i,-1)\in M \mid i\in\mathbb{Z}\}$. Since $S$ is not saturated, by \cite[Exercice 7.15]{MiSt05}  (see also \cite[Lemma 2]{BoGa21}), we can choose  $(m_1,m_2)\in S'\setminus S$ such that $(j,m_2+1) \in S$ for all $j\in \ZZ$. By Proposition~\ref{inclusion}, we have $\mathcal{R}(S)\subset\mathcal{R}(S')$. Take now $\alpha=(i,-1)\in \mathcal{R}(S')$ be any root of $S'$. Now the element $m=(m_1-i,m_2+1)$ is in $S$ but $m+\alpha$ is not in $S$. This yields $\mathcal{R}(S)=\emptyset$.
\end{proof}

\begin{remark} \label{torus-no-nonnormal}
    If $X_\sigma$ is the normalization of $X_S$ and $X_\sigma$ is isomorphic to the algebraic torus, then $X_S=X_\sigma$. Indeed, in this case $X_\sigma=\spec \KK[M]$ with $M\simeq \ZZ^2$ which is a group. Assume that $S$ is a monoid whose saturation is $M$. By \cite[Exercise~7.15]{MiSt05}, we have $S=M$.
\end{remark}

\begin{lemma} \label{cor-roots}
Let $X_\sigma$ be a normal affine toric surface. Then, there exists a non-normal affine toric surface $X_S$ whose normalization is $X_\sigma$ with $\mathcal{R}(\sigma)=\mathcal{R}(S)$ if and only if $X_\sigma$ is different from $\GM^2$, $\GM\times \AF^1$ and $\AF^2$.
\end{lemma}

\begin{proof}
Let $\sigma\subset N_\QQ\simeq \QQ^2$ be a strongly convex polyhedral cone. If $X_\sigma$ is different from $\GM^2$, $\GM\times \AF^1$ and $\AF^2$, then the lemma follows from Proposition~\ref{roots: non normal}  taking $S=(\sigma^\vee\cap M)\setminus \mathcal{H}$. If $X_\sigma=\GM^2$, then by Remark~\ref{torus-no-nonnormal}, there are no non-normal affine toric varieties with normalization $X_\sigma$ so the lemma holds trivially. If $X_\sigma=\GM\times \AF^1$, then $X_\sigma$ admits a $\GA$-action by translation in the $\AF^1$ factor, so $\mathcal{R}(\sigma)\neq \emptyset$ while $\mathcal{R}(S)=\emptyset$ by Lemma~\ref{no-GA}. Finally, if $X_\sigma=\AF^2$ the lemma follows from \cite[Proposition~6.2]{liendo2018characterization}.
\end{proof}

To achieve the proof of Theorem~\ref{main-surface}, we need the following lemma borrowed from \cite{A13}.

\begin{lemma} \label{AZ-aut-toric-surface}
The automorphism group of a non-degenerated normal affine toric surface $X_\sigma$ is generated by the torus and the images in $\Aut(X_\sigma)$ of all the $\GA$-actions corresponding to roots $\alpha\in \mathcal{R}(\sigma)$. In particular, if $X_S$ is a non-normal affine toric surface $X_S$ whose normalization is $X_\sigma$ with $X_\sigma$ non-degenerated and $\mathcal{R}(\sigma)=\mathcal{R}(S)$, then $\Aut(X_\sigma)=\Aut(X_S)$.
\end{lemma}

\begin{proof}
The second assertion follows directly from the first by Lemma~\ref{lifting-to-normalization}. The first assertion follows from \cite[Lemma~4.3]{A13}. Indeed, in their notation, $\Aut(X_\sigma)$ is a quotient of $\mathcal{N}_{d,e}$ and by \cite[Proposition~4.4 and Lemma~4.5]{A13} $\mathcal{N}_{d,e}$ consists of de Jonqui\`eres transformations and de Jonqui\`eres transformations consist precisely of compositions of the  images in $\Aut(X_\sigma)$ of the acting torus and of all the $\GA$-actions corresponding to roots $\alpha\in \mathcal{R}(\sigma)$.
\end{proof}

We can now prove our main result in Theorem~\ref{main-surface}.

\begin{proof}[Proof of Theorem~\ref{main-surface}]

If $X_\sigma$ is different from $\GM^2$, $\GM\times \AF^1$ and $\AF^2$, the theorem follows from Lemma~\ref{cor-roots} and Lemma~\ref{AZ-aut-toric-surface} taking for $X_S$ the affine toric variety given by the monoid $S=(\sigma^\vee\cap M)\setminus\mathcal{H}$. If $X_\sigma=\AF^2$ the theorem follows from \cite[Proposition~6.2]{liendo2018characterization}. 

If $X_\sigma=\GM^2$, then by Remark~\ref{torus-no-nonnormal}, there are no non-normal affine toric varieties with normalization $X_\sigma$ so the theorem holds trivially. 

If $X_\sigma=\GM\times \AF^1$, we claim that $\Aut(X_\sigma)=(\ZZ/2\ZZ\ltimes \GM)\ltimes \left(\KK[t,t^{-1}]^*\ltimes\KK[t,t^{-1}]\right)$, where the factor $\ZZ/2\ZZ\ltimes \GM$ corresponds to the automorphism of $\GM$ and the factor $\KK[t,t^{-1}]^*\ltimes\KK[t,t^{-1}]$ corresponds to automorphisms of $\GM\times \AF^1$ that preserve the $\AF^1$-fibration fiberwise. Indeed, let $\varphi\in \Aut(X_\sigma)$. The semidirect product structure comes from the fact that the composition $\pr_1\circ\varphi\colon \GM\times\AF^1\to \GM$ is constant on the fibers $\{t\}\times \AF^1$ with $t \in \GM$ since there are no non-constant maps $\AF^1\to \GM$. On the other hand, $\KK[S]$ does not have any homogeneous locally nilpotent derivation by Lemma~\ref{no-GA} and so $X_S$ it does not admit any $\GA$-action by Theorem~\ref{LND-decomposition}. Hence, by Lemma~\ref{lifting-to-normalization}, we obtain that $\Aut(X_S)\subseteq (\ZZ/2\ZZ\ltimes \GM)\ltimes\KK[t,t^{-1}]^*$ and so $\Aut(X_\sigma)$ is not isomorphic $\Aut(X_S)$ ending the proof.
\end{proof}

\bibliographystyle{alpha}
\bibliography{ref}

\end{document}